\documentclass[12pt]{amsart}

\usepackage{amssymb}

\usepackage[inline]{enumitem}

\makeatletter
\@namedef{subjclassname@2020}{%
  \textup{2020} Mathematics Subject Classification}
\makeatother

\newtheorem{theorem}{Theorem}[section]
\newtheorem{corollary}[theorem]{Corollary}
\newtheorem{lemma}[theorem]{Lemma}
\newtheorem{proposition}[theorem]{Proposition}

\theoremstyle{definition}
\newtheorem{definition}[theorem]{Definition}

\newtheorem{claim}[theorem]{Claim}

\usepackage[breaklinks,pdfstartview=FitH]{hyperref}

\hypersetup{
     colorlinks   = true,
     citecolor    = black
}
\hypersetup{linkcolor=black}
\hypersetup{urlcolor=black}

\newcommand{\p}{\psi}

\newcommand{\T}{\mathsf{T}}
\newcommand{\Ntree}{\mathsf{N}}
\newcommand{\Stree}{\mathsf{S}}

\newcommand{\e}{\varepsilon}

\newcommand{\rootv}{\mathsf{r}}

\DeclareMathOperator{\diam}{diam}

\renewcommand{\emptyset}{\varnothing}

\newcommand{\mlevel}[1]{\ell({#1})}
\newcommand{\rep}[1]{\mathrm{x}(#1)}
\newcommand{\repT}[2]{\mathrm{x}_{#1}(#2)}
\newcommand{\trep}[1]{\tilde{\mathrm{x}}(#1)}

\newcommand{\pb}{\diamondsuit}
\newcommand{\tvb}{\tilde \heartsuit}
\newcommand{\Image}{\mathrm{Im}}
\newcommand{\bQ}{Q^c}
\newcommand{\bbQ}{\overline{Q^c}}

\newcounter{EnumResume}

\begin{document}




\title{Dvoretzky-type theorem for Ahlfors regular spaces}

\author[M. Mendel]{Manor Mendel}
\address{Department of Mathematics and Computer Science\\ The Open University of Israel\\ 1 University Road\\ Raanana 43107, Israel}
\email{manorme@openu.ac.il}

\date{}



\begin{abstract}
It is proved that for any $0<\beta<\alpha$, any bounded Ahlfors $\alpha$-regular space contains a $\beta$-regular compact subset that
embeds biLipschitzly in an ultrametric with distortion at most
$O(\alpha/(\alpha-\beta))$. 
The bound on the distortion is asymptotically tight
when $\beta\to \alpha$.
The main tool used in the proof is a regular form of the ultrametric skeleton theorem.
\end{abstract}

\maketitle

\setcounter{tocdepth}{1} 
\tableofcontents

\section{Introduction} \label{sec:intro}

Fix a metric space $(X,d)$, a point $x\in X$ and a radius $r\in [0,\infty)$. The corresponding closed ball is denoted $B(x,r)=B_d(x,r)=\{y\in X:\ d(y,x)\le r\}$, and the corresponding open ball is denoted $B_d^\circ(x,r)=\{y\in X:\ d(y,x)< r\}$. 
A complete metric space $(X,d)$ is called \emph{Ahlfors $\alpha$-regular} (or $\alpha$-regular for short), 
if there exists a Borel measure $\mu$ such that for all $x\in X$ and $r\in(0,\diam(X))$,
\begin{equation} \label{eq:def:ahlfors-regular}
cr^\alpha \le \mu(B_d(x,r)) \le C r^\alpha .
\end{equation}
Here $C\ge c>0$ are independent of $x$ and $r$. 
Ahlfors $\alpha$-regular space $X$ has, in particular,  Hausdorff dimension $\dim_H(X)=\alpha$.
For more information on Ahlfors regular spaces and their importance, see~\cite{Hutchinson,DS-fractals,Heinonen}.

An ultrametric space is a metric space $(U,\rho)$ satisfying the strengthened triangle inequality $\rho(x,y)\le \max\{\rho(x,z),\rho(y,z)\}$ for all $x,y,z\in U$. 
Saying that $(X,d)$ embeds (biLipschitzly) with distortion $D\in [1,\infty)$ into an ultrametric space means that there exists an ultrametric $\rho$ on $X$ satisfying $d(x,y)\le \rho(x,y)\le Dd(x,y)$ for all $x,y\in X$. 
The \emph{ultrametric distortion} of $X$ is the infimum over $D$ for which $X$ embeds in an ultrametric with distortion at most $D$.

In this paper we study  regular (approximate) ultrametric subsets of Ahlfors regular spaces.
Arcozzi et.~al.~\cite[Theorem~1]{arcozzi2019ahlfors} proved that for  
every 
$0<\beta<\alpha$, any  bounded
Ahlfors $\alpha$-regular space $X$ contains a 
$\beta$-regular subset $Y$. 
Matilla and Saaranen~\cite[Theorem~3.1]{MS-maps} (see also \cite[Corollary~5.2]{JJK}) showed that $Y$ can be chosen to be both $\beta$-regular and biLipschitz embeddable in an ultrametric.
In their proof, the ultrametric distortion of $Y$ is bounded above by
$\exp(O(\alpha/(\alpha-\beta)))$.%
\footnote{%
The embedding of $Y$ in~\cite{MS-maps} is into a Cantor set of $[0,1]^n$ for some $n>\beta$,
and the distortion is $\sqrt{n}\exp(O(\alpha/(\alpha-\beta)))$. 
Their embedding is factored through a $\beta$-regular ultrametric. 
The $\sqrt{n}$ factor is  the distortion of the straightforward embedding of the ultrametric in an $n$-dimensional Cantor set, which is
irrelevant in our setting. 
However, the exponential dependence on $\alpha/(\alpha-\beta)$ seems inherent to their approach.
As aside, we note that bounded $\beta$-regular ultrametrics can be embedded in  $n$-dimensional Cantor set with a constant distortion, 
provided that $n\ge C \beta$, for some a constant $C>1$.
This can be done using binary error-correcting codes as in~\cite[Proposition~3]{LMM-hausdorff}.}
Here we prove a similar result with an exponentially improved bound on the 
ultrametric distortion.

\begin{theorem}
\label{thm:cantor-ahlfors}
For every $0<\beta<\alpha$, any bounded
Ahlfors $\alpha$-regular space $X$ contains a  $\beta$-regular compact subset $Y$ whose ultrametric distortion is $O(\alpha/(\alpha-\beta))$.
\end{theorem}

When $\beta\to \alpha$ the  bound on the distortion in the Theorem~\ref{thm:cantor-ahlfors} is asymptotically tight. 
A tight example is given in~\cite[Theorem~1.4]{MN-hausdorff}.
Specifically, for any $\alpha>0$ a compact metric $X_\alpha$ 
is constructed  with the following properties. 
Its Hausdorff dimension is $\dim_H(X_\alpha)=\alpha$, and
for any subset $S\subseteq X_\alpha$,  the ultrametric distortion (indeed, even the Euclidean distortion) of $S$ 
is $\Omega(\alpha/(\alpha-\beta))$, where $\beta=\dim_H(S)$.
Observing the space $X_\alpha$ from~\cite{MN-hausdorff}, it is clear that it is compact, self-similar and it satisfies the open-set condition of~\cite{Hutchinson}. Therefore it is Ahlfors $\alpha$-regular (see~\cite[Theorem~5.3(1)]{Hutchinson}). 
Furthermore, when $\alpha=n\in\mathbb{N}$ is a natural number,
$[0,1]^n\subset \mathbb{R}^n$ is also an asymptotically tight example: 
It is proved in~\cite{LMM-hausdorff} that for any 
$S\subseteq [0,1]^n$ of Hausdorff dimension $\beta=\dim_H(S)$, the ultrametric distortion of $S$ is at least
$0.25\cdot n/(n-\beta)$.

Results similar to Theorem~\ref{thm:cantor-ahlfors} were previously obtained in other settings and were called
\emph{Dvoretzky-type theorems for metric spaces} and \emph{matric Ramsey theory}, see~\cite[\S8~\&~\S9]{Naor-Ribe} and references therein. 
In particular, it is proved in~\cite{MN-hausdorff,MN-skeleton} that 
for every  compact metric space $X$ of
Hausdorff dimension $\alpha=\dim_H(X)$ and 
any $0<\beta<\alpha$, there exists a closed
subset $Y\subset X$ and a Borel measure $\nu$ supported on $Y$
such that $Y$ embeds in an ultrametric with distortion
$O(\alpha/(\alpha-\beta))$, and $\nu$ satisfies
\(
\nu(B(x,r))
\le 
C r^\beta
\)
for any $x\in Y$ and $r\in(0,\diam_d(Y))$
(which implies that $\dim_H(Y)\ge \beta$).
Here we  amend the arguments from~\cite{MN-hausdorff,MN-skeleton,mendel2021simple} so that
in the Ahlfors regular setting the measure $\nu$ would also satisfy
\(
\nu(B(x,r)) 
\ge  c r^\beta
\).

The tool we use 
is a regular form of the ultrametric skeleton theorem~\cite{MN-skeleton}, which apply more broadly to \emph{doubling spaces}.
A metric space $(X,d)$ is called \emph{$\lambda$-doubling} if any
bounded subset $Z\subseteq X$ 
can be covered by at most $\lambda$ subsets of diameter at most 
$\diam_d(Z)/2$. 
$X$ is called \emph{doubling} if it is $\lambda$-doubling for some $\lambda\in\mathbb N$.  
A Borel measure $\sigma$ on $(X,d)$ is called
$\kappa$-\emph{doubling measure} if 
$\sigma(B_d(x,2r))\le \kappa\cdot \sigma(B_d(x,r))$, for any $x\in X$ and $r>0$.
A complete metric space is doubling if and only if it has a doubling measure~\cite{VK,LS}.
Observe that an Ahlfors $\alpha$-regular space is doubling
since the measure $\mu$ from~\eqref{eq:def:ahlfors-regular}  
is a $({C2^\alpha}/{c})$-doubling measure.

 \begin{theorem}[Regular ultrametric skeleton for doubling spaces]
 \label{thm:um-skeleton-regular}
 Let $(X,d)$ be a compact $\lambda$-\emph{doubling} metric space, and let $\mu$ be a Borel probability measure on $X$.
 Then for every $t\in\mathbb N$ there exists a compact subset $S\subseteq X$  and
 a Borel probability measure $\nu$ supported on $S$ satisfying:
 \begin{itemize}
 \item  The ultrametric distortion of $S$ is at most $16t$.
 \item For every $x\in X$ and $r\in [0,\infty)$,
 \begin{equation}\label{eq:measure growth-regular}
 \nu\left(B_d(x,r)\right)\le  \lambda^{2/t} \cdot \mu(B_d(x,C_1 t  r)^{1-1/t}.
 \end{equation}
 \item For every $y\in S$ and $r\in [0,\infty)$, 
 there  exists $x\in X$ such that $B_d(x,c_1 r/t)\subset B_d(y,r)$ and
 \begin{equation}\label{eq:measure-shrink-regular}
 \nu\left(B_d(y,r)\right) \ge (\lambda^{-2/t}/2) \cdot \mu(B_d(x,c_1 r/t))^{1-1/t} .
 \end{equation}
 \end{itemize}
 Here $C_1,c_1>0$ are universal constants.
 \end{theorem}
 
Informally speaking, the theorem above constructs for every doubling metric space and Borel measure $\mu$, an approximate 
ultrametric subset (a ``skeleton") with a measure $\nu$ that behaves ``similarly to $\mu^{1-1/t}$". 
The original ultrametric skeleton theorem~\cite{MN-skeleton} has applications for 
online algorithms, data-structures, probability theory, and geometric measure theory. 
See~\cite[\S1]{MN-skeleton} for more details.
Theorem~\ref{thm:um-skeleton-regular}, when compared to 
the original ultrametric skeleton, has an additional
lower bound~\eqref{eq:measure-shrink-regular} on $\nu$
in the conclusion,
but it also has an additional doubling condition in the assumptions. 
See Section~\ref{sec:remarks} for remarks about the necessity of the doubling assumption.

The construction of the ultrametric skeleton here follows 
the construction in~\cite{mendel2021simple}, which
in turn uses Bartal's Ramsey decomposition lemma~\cite{bartal2021advances} as a key tool. 
Since we will use Bartal's Ramsey decomposition lemma~\cite{bartal2021advances}  in slightly more general form than in~\cite{bartal2021advances,mendel2021simple}, we rephrase and reprove it in Section~\ref{sec:bartal}.
The proof of Theorem~\ref{thm:um-skeleton-regular} is detailed in Section~\ref{sec:um-skeleton-regular},
after introducing in Section~\ref{sec:net-trees} an auxiliary structure of hierarchical nets which is needed to ensure the existence of the ``small balls" in~\eqref{eq:measure-shrink-regular}.
Finally, the proof of Theorem~\ref{thm:cantor-ahlfors}, 
a rather straightforward corollary of Theorem~\ref{thm:um-skeleton-regular},
is presented in Section~\ref{sec:dvo-proof}.

\section{Bartal's Ramsey decomposition lemma}
\label{sec:bartal}

Fix a compact metric measure space $(X,d,\mu)$, and $\Delta>0$.  
For a Borel subset $A\subseteq X$, define
\begin{equation} \label{eq:mu-star}
\mu^\Delta(A)=\sup_{a\in A}\mu(B_d(a, \Delta /4)\cap A).
\end{equation}  

Before delving into the details of Bartal's lemma,
we give a quick and intuitive explanation of it and its use 
in the proof of Theorem~\ref{thm:um-skeleton-regular}.
See also \cite[\S1.2]{mendel2021simple} for a similar 
explanation in a somewhat simpler setting.
Roughly speaking, Corollary~\ref{cor:bar-ramsey-decomp-2}  below
states that any compact metric measure space $(X,d,\mu)$ of diameter $\Delta=\diam(X)$, contains two non-empty closed subsets $P,Q\subset X$ that are separated, i.e., $d(P,Q)\ge \Delta/(8t)$
for some fixed $t\in\{2,3,\ldots\}$, and
\begin{equation} \label{eq:bar-explanation}
\frac{\mu(P)}{\mu^{\Delta/2}(P)^{1/t}} + \frac{\mu(Q)}{\mu^\Delta(Q)^{1/t}} \ge \frac{\mu(X)}{\mu^\Delta(X)^{1/t}}.
\end{equation}
Recursive  applications of~\eqref{eq:bar-explanation} creates a hierarchy of subsets.
The separation property implies that the "surviving" set of points, $C$, forms a Cantor-like subset.
I.e., $C$ is biLipschitz equivalent to an ultrametric,
and its ultrametric distortion is $8t$. 
In doubling space, $\mu^\Delta(X)\asymp \mu(X)$,
and therefore~\eqref{eq:bar-explanation} implies, 
roughly speaking, that
$\xi=\mu/\mu^\Delta$ is a sub-measure controlled from above by $\mu^{1-1/t}$. 
By ``pruning" the set $C$ a bit further we can obtain a subset $S\subset C$ on which $\xi$ is additive.
To control $\xi$  from below using $\mu^{1-1/t}$, 
we should ensure that the subsets $P$ and $Q$ in~\eqref{eq:bar-explanation} and all the subsets created from iterative applications of~~\eqref{eq:bar-explanation} contain ``small" balls of $X$. 
To achieve that, we do not actually  apply~\eqref{eq:bar-explanation} directly on subsets of $X$, 
and instead opt to apply  it on (a hierarchy of) metric nets  $\mathcal N$ of $X$, 
where each of the net's points represents its Voronoi cell.
Since the diameter of $\mathcal N$ may be slightly smaller than the diameter of $X$ when applying~\eqref{eq:bar-explanation} on $\mathcal N$, we actually use $\Delta\approx\diam(\mathcal N)+``\e"$ instead of $\Delta=\diam(\mathcal N)$, which explains why $\Delta$ is a parameter in~\eqref{eq:bar-explanation}.

We next continue with rigorous definitions and arguments.
The following properties of $\mu^\Delta$ are straightforward (for a proof of~\eqref{eq:mu-mu*} see~\cite{mendel2021simple}):
\begin{proposition} 
\label{prop:mu-Delta-prop}
Let $0<\delta\le \Delta$, and $A\subseteq C\subseteq X$.
Then $\mu^\delta(A)\le \mu^\Delta(C)$.
If $X$ is $\lambda$-doubling, and $\Delta\ge\diam(A)$, then
\begin{equation}
\label{eq:mu-mu*}
 \mu^\Delta(A)\le \mu(A) \le \lambda^{2}\mu^{\Delta}(A).
\end{equation}
\end{proposition}

\begin{lemma}
\label{lem:bar-ramsey-decomp}
Let $(X,d)$ be a compact metric space, and let $\mu$ be a finite Borel measure on $X$. 
For any  compact subset $\emptyset\ne Z\subset X$,
$0<\Delta < 2\diam_d(Z)$ and integer $t\in\{2,3,\ldots\}$, there exist non-empty disjoint and compact subsets $P,Q\subset Z$ that, denoting $\bQ=Z\setminus Q$, satisfy
$d(P,Q)\ge \Delta /(8t)$,
$\diam_d(\bQ)\le \Delta/2$,
$\diam_d(P)\le \bigl( \frac12 - \frac1{4t}\bigr) \Delta$, and
\begin{equation} \label{eq:bar-ramsey-decomp}
\mu(P) \ge \mu (\bQ)\cdot \left(\frac{\mu^{\Delta/2} (\bQ)}{\mu^\Delta(Z)}\right)^{1/t},
\end{equation}
where we use the convention $0/0=0$ in~\eqref{eq:bar-ramsey-decomp}.
\end{lemma}
\begin{proof}
The statement of the Lemma is vacuous when $\mu(Z)=0$, so we
assume that $\infty>\mu(Z)>0$.
With the convention $0/0=0$, let $x\in Z$ be a point that maximizes%
\footnote{Indeed, the maximum here and below exists. See~\cite[Remark~3]{mendel2021simple} for a straightforward proof.}
\[
\frac{\mu (B(x,\Delta/8)\cap Z)} {\mu(B^o(x,\Delta/4)\cap Z)}.
\]
With this choice, $\mu (B(x,\Delta/8)\cap Z)>0$.
For $i\in\{0,1,\ldots,t-1\}$, let $H_i=B(x,(1+i/t)\Delta/8)\cap Z$, and also define
$H_t=B^o(x,\Delta/4)\cap Z$.
Clearly there exists $i\in\{1,\ldots,t\}$ for which
\begin{equation}\label{eq:light-ring}
\mu(H_i)\le \mu(H_{i-1}) \cdot \biggl(\frac{\mu(H_t)}{\mu(H_0)}\biggr)^{1/t}.
\end{equation}
We then set $P=H_{i-1}$, 
$\bQ=B^o_d(x,(1+i/t)\Delta/8)\cap Z$, 
and $Q=Z\setminus \bQ$.
Observe that 
$P$ and $Q$ are closed, $\mu(P)>0$,
and
$H_i\supseteq \bQ\supseteq P$.
By the triangle inequality, for every $a\in P$, and $b\in Q$, $d(a,b)\ge d(b,x)-d(a,x)\ge \Delta/(8t)$.
The measure satisfies
\begin{equation} \label{eq:light-ring-2}
\mu(P) \stackrel{\eqref{eq:light-ring}}{\ge} 
\mu(\bQ) \cdot \biggl(\frac{\mu(B(x,\Delta/8)\cap Z)}{\mu(B^o(x,\Delta/4)\cap Z)}\biggr)^{1/t}.
\end{equation}

Let $u\in \bbQ\subseteq B(x,(1+i/t)\Delta/8 )$ the point that maximizes
$\mu(B(u,\Delta/8)\cap \bbQ)$,
where $\bbQ$ is the topological closure of $\bQ$.
Since 
$B(u,\Delta/8)\cap \bbQ\subseteq B(u, \Delta/8 )\cap Z$,  we have
$\mu^{\Delta/2}(\bbQ)\le \mu(B(u,\Delta/8)\cap Z)$.
Also,
$\mu^\Delta(Z)\ge \mu(B^o(u,\Delta/4)\cap Z)$.
From the definition of $x$
we therefore have
\[
\frac{\mu(B(x,\Delta/8)\cap Z)}{\mu(B^o(x,\Delta/4)\cap Z)}
\ge 
\frac{\mu(B(u,\Delta/8)\cap Z)}{\mu(B^o(u,\Delta/4)\cap Z)}
\ge 
\frac{\mu^{\Delta/2}(\bbQ)}{\mu^\Delta(Z)} 
\ge \frac{\mu^{\Delta/2}(\bQ)}{\mu^\Delta(Z)} . 
\]
Applying the above inequality to~\eqref{eq:light-ring-2},
we obtain~\eqref{eq:bar-ramsey-decomp}. 
\end{proof}

We use Lemma~\ref{lem:bar-ramsey-decomp} via the following corollary.

\begin{corollary}\label{cor:bar-ramsey-decomp-2}
Let $(X,d)$ be a compact metric space and let $\mu$ be a Borel probability measure on $X$.
For any closed subset $Z\subset X$, $\mu(Z)>0$, and integer $t\in\{2,3,\ldots\}$, there exist closed subsets $P, Q\subseteq Z$ such that 
$\mu(P)>0$,
$d(P,Q)\ge \Delta /(8t)$,
$\diam_d(Z\setminus Q)\le \Delta/2$,
$\diam_d(P)\le \bigl( \frac12 - \frac1{4t}\bigr) \Delta$, and
\begin{equation} \label{eq:bar-ramsey-decomp-2}
\frac{\mu(P)}{\mu^{\Delta/2}(P)^{1/t}} + \frac{\mu(Q)}{\mu^\Delta(Q)^{1/t}} \ge \frac{\mu(Z)}{\mu^\Delta(Z)^{1/t}}.
\end{equation}
where $\mu^\Delta$ is defined in~\eqref{eq:mu-star}, 
and with the convention $0/0=0$.
\end{corollary}
\begin{proof}
Denote $\bQ=Z\setminus Q$.
Applying Lemma~\ref{lem:bar-ramsey-decomp} on $Z$, 
\begin{multline*}
\frac{\mu(P)}{\mu^{\Delta/2}(P)^{1/t}} + \frac{\mu(Q)}{\mu^\Delta(Q)^{1/t}} \ge
\frac{\mu(P)}{\mu^{\Delta/2}(\bQ)^{1/t}} + \frac{\mu(Q)}{\mu^{\Delta}(Z)^{1/t}} \\
\stackrel{\eqref{eq:bar-ramsey-decomp}}\ge
\frac{\mu(\bQ)}{\mu^{\Delta/2}(\bQ)^{1/t}} \cdot \frac{\mu^{\Delta/2}(\bQ)^{1/t}}{\mu^\Delta(Z)^{1/t}}
+  \frac{\mu(Q)}{\mu^\Delta(Z)^{1/t}}
= \frac{\mu(Z)}{\mu^\Delta(Z)^{1/t}}. 
\end{multline*}
The first inequality above is obtained using Proposition~\ref{prop:mu-Delta-prop}
and recalling that $P\subseteq Q^c$, and $Q\subseteq Z$.
\end{proof}

\section{Net-trees and ultrametrics in compact spaces}
\label{sec:net-trees}

As explained at the beginning of Section~\ref{sec:bartal}, Bartal's Ramasey decomposition lemma will be employed on
an auxiliary hierarchy of \emph{metric nets}, the existence of which should be considered a folklore.
It is mostly similar to (the easy part of) Christ's dyadic decomposition~\cite{Christ} for  spaces with doubling measure, or to the \emph{net-tree}~\cite{HPM06} for finite spaces. 
For completeness, we provide here a self-contained
treatment for compact spaces. 

\begin{definition}[Rooted trees]
\label{def:tree}
A rooted tree $\T$ is a set of vertices with a distinguished vertex $\rootv=\rootv_\T \in \T$ called the \emph{root}. 
Every vertex $u\in \T\setminus \{\rootv\}$, except $\rootv$, has a unique \emph{parent} $\p(u)\in\T$.
The $k$-th ancestor $\p^{(k)}(u)$(if exists) of a vertex $u$ is defined inductively, $\p^{0}(u)=u$, and 
$\p^{(k+1)}(u)=\p(\p^{(k)}(u))$. 
The set of ancestors of $u$ is written as $\p^{(*)}(u)=\{\p^{(0)}(u),\p^{(1)}(u),\p^{(2)}(u),\ldots\}$.
For every vertex $u\in\T$, the sequence of ancestors $u=\p^{(0)}(u),\p^{(1)}(u),\ldots$ is finite and ends with the root $\rootv$. 
If $u=\p(v)$ then $v$ is called a child of $u$. 
The set of children of a vertex $u$ is denoted $\p^{-1}(u)$ and must be finite and non-empty (i.e., there are no leaves).
For $u\in \T$, we denote by $\T_u=\{v\in\T: u\in\p^{(*)}(v)\}$ the set of (weak) descendants of $u$ in $\T$.

For every two vertices $u,v\in\T$ we denote by $u\wedge v\in\T$ their \emph{least common ancestor}, i.e., 
$u\wedge v\in \p^{(*)}(u)\cap \p^{(*)}(v)$, and if $w\in \p^{(*)}(u)\cap \p^{(*)}(v)$ then $w\in \p^{(*)}(u\wedge v)$.

A \emph{branch}  is an infinite sequence of vertices $b=(\rootv=u_0,u_1,u_2,\ldots)$ that begins with the root
and has $u_i=\p(u_{i+1})$. 
The least common ancestor extends to branches: for two branches $x\ne y$,
$x\wedge y$ is the deepest common vertex $v\in x\cap y$.
The \emph{tree boundary} $\tilde\partial(u)=\tilde\partial_\T (u)$ of a vertex $u\in \T$ is defined
as the set of branches that contain $u$.
\end{definition}

For a function $f:\T\to X$,
and a branch $b=(v_0,v_1,\ldots)\in\tilde\partial_\T (\rootv)$, we sometimes use the notation
$\lim_{v\to b}f(v)=\lim_{i\to\infty}f(v_i)$.

\begin{definition}[subtree]
Fix a tree $\T$ and vertex $u\in \T$.
A subset $\Stree\subset\T$, is called a subtree of $\T$ rooted at $u$ if the induced parent relation 
of $\T$ on $\Stree$ forms a rooted tree (according to Definition~\ref{def:tree}) whose root is $u$. 
For example, $\T_u$ is a subtree of $\T$ rooted in $u\in\T$. 
\end{definition}

\begin{definition}[Net-tree] \label{def:net-tree}
Let $(X,d)$ be a metric space, and $\kappa\ge 1$.
A $\kappa$-\emph{net-tree} over $X$ is a
rooted tree $\Ntree$  
(in the sense of Definition~\ref{def:tree})
with the root being $\rootv=\rootv_\Ntree$,
together with:
\begin{description}
\item [\quad Labels for vertices] $\Delta:\Ntree \to [0,\infty)$.
\item [\quad Representative points for vertices]
$\rep{\cdot}:\Ntree\to X$.
\end{description}

The net-tree should satisfy the following conditions:
\begin{description}
\item [\quad Monotone labels]
$\Delta(u)\le \Delta(\p(u))$ for every $u\in \Ntree\setminus\{\rootv\}$.
\item [\quad Vanishing labels]
 $\lim_{v\to b}\Delta(v)=0$, for every $b\in\tilde\partial(\rootv)$.

\item[\quad Covering] 
For every vertex $u\in \Ntree$,
\[
\{\rep{w}: w\in \T_u\}\subseteq B_d\bigl({\rep{u}, \Delta(u) }\bigr).
\]

\item [\quad Packing] For every 
$u,v\in \Ntree$, if $u\wedge v \notin \{u,v\}$ (i.e., no ancestor/descendant relation) and $\Delta(v)\le \Delta(u)$,
then 
\[ B_d(\rep{u},\Delta(\p(u))/\kappa)\cap \{\rep{w}: w\in \T_v\}=\emptyset.\]
\end{description}
We call $\Ntree$ a \emph{net-tree} over $X$ if it is $\kappa$-net-tree over $X$ for some $\kappa\in[1,\infty)$.
\end{definition}

For a vertex $u\in\Ntree$ with $\Delta(u)>0$ and $\delta\in(0,\Delta(u)]$, 
we further define the $\delta$-\emph{descendants} of $u$ as
\[
\tvb(u,\delta)=\tvb_\Ntree(u,\delta)= 
\begin{cases}
\{u\} & \delta=\Delta(u)\\
\{
v\in\T_u:\;  \Delta(v)\le \delta<\Delta(\p(v))\}
& \delta\in(0,\Delta(u)).
\end{cases}
\]

By a straightforward induction $\tvb(u,\delta)$ is finite.

\begin{proposition}
Let $\Ntree$ be a net-tree over $(X,d)$.
For every branch $b\in\tilde\partial(\rootv)$,  
$\lim_{v\to b} \rep{v}$ exists.
\end{proposition}
\begin{proof}
Fix a branch $b=(\rootv=v_0,v_1,v_2,\ldots)\in\tilde\partial(\rootv)$. 
Let $D_n=\bigcap_{i=0}^n B_d(\rep{v_i},\Delta(v_i))$.
$(D_n)_{n\ge 0}$
is a decreasing sequence of compact subsets with diameters that tend to~$0$. 
By the covering property, $\rep{v_n}\in D_n\ne \emptyset$.
Hence, by Cantor intersection theorem, 
$\bigcap_{n=0}^\infty D_n=\bigcap_{i=0}^{\infty} B_d(\rep{v_i},\Delta(v_i))$ 
is a singleton which is the limit of $\rep{v_n}$.
\end{proof}

The above proposition allows us to extend the notion of representative points to branches.

\begin{definition}[Boundary and Surjection]
Let $\rep{\cdot}:\tilde \partial_\Ntree(\rootv)\to X$, defined by $\rep{b}=\lim_{v\to b}\rep{v}$.
The \emph{boundary} of a vertex $v$,
$\partial(v)=\partial_\Ntree(v)=\{\rep{b}:\; b\in\tilde\partial(v)\}$, is the set of points represented by the branches that contain $v$.
The image of the net-tree $\Ntree$ is defined as the boundary of the root $\Image(\Ntree)=\partial_\Ntree(\rootv_\Ntree)$.
A net-tree $\Ntree$ over a metric space $(X,d)$ is called \emph{surjective} if $\Image(\Ntree)=X$.
\end{definition}

\begin{proposition}\label{prop:net-tree-prop}
Further properties of a net-tree $\Ntree$ over a compact space $(X,d)$:
\begin{enumerate}[label=(\Alph*)]
\item \label{it:nt-prop-A}
For every $u\in \Ntree$, $\partial(u)\subseteq  B_d\bigl({\rep{u}, \Delta(u) }\bigr)$.
\item \label{it:nt-prop-B}
A subtree $\Stree \subset \Ntree$ is also a net-tree.
\item \label{it:nt-prop-C}
For every $u\in \Ntree$, $\partial(u)$ is compact.
In particular, $\Image(\Ntree)$ is compact.
\end{enumerate}
\end{proposition}
\begin{proof}
Item~\ref{it:nt-prop-A}: Fix $b\in \tilde\partial(u)$. By the covering property, for almost all $v\in b$, $\rep{v}\in B_d(\rep{u},\Delta(u))$, and since this ball is closed, we conclude that it is also contains
$\lim_{v\to b}\rep{v}$.

Item~\ref{it:nt-prop-B}: A subtree is also a net-tree since the four required properties: monotone and vanishing labels, covering, and packing, are closed for subtrees.

Item~\ref{it:nt-prop-C}: Next we prove that $\partial(u)$ is compact.
Assume $\partial(u)$ is infinite (otherwise, it is trivially compact). 
Fix an infinite $A\subseteq \partial(u)$. We should prove that $A$ has an accumulation point in $\partial(u)$. 
To achieve it, we construct an infinite non-increasing sequence of infinite subsets $A=A_1\supseteq A_2\supseteq A_3\supseteq \ldots$ and an infinite  sequence of vertices $u=v_1,v_2, v_3,\ldots$ in $\Ntree$ satisfting 
$A_i\subseteq \partial(v_i)$, and $v_{i+1}\in\p^{-1}({v_{i}})$, 
The construction is by induction.
The base case $v_1=u$, $A_1=A\subseteq\partial(u)$ holds by assumption. 
Assume we have already defined $A_1,\ldots, A_n$ and $v_1,\ldots,v_n$ to satisfy the above when $i\in\{1,\ldots,n-1\}$. 
Since $\p^{-1}(v_n)$ is
finite
and
\[
A_n=A_n\cap \partial(v_n)=\bigcup_{w\in\p^{-1}(v_n)}
(A_n\cap\partial(w))
\]
is infinite,
there must be $w\in \p^{-1}(v_n)$ for which
$A_n\cap \partial(w)$ is infinite. 
Define $v_{n+1}=w$ and
$A_{n+1}=A_n\cap\partial(v_{n+1})$.
The sequence $(v_n)_n$ is a
suffix of a unique branch $b\in\tilde\partial(u)$.
For every $n\in\mathbb N$ we have, by the covering property,
\[
\emptyset\ne A_n\subseteq A\cap \partial(v_n)\subseteq B(\rep{v_n},\Delta(v_n))\ni\rep{b}.
\]
Hence, $B_d(\rep{b},2\Delta(v_n))\cap A \supseteq A_n\ne \emptyset$.  Since $\Delta(v_n)\to 0$, it means that 
$\rep{b}\in\partial(u)$ is an accu\-mu\-lation point of~$A$.
\end{proof}

The mapping $\rep{\cdot}:\tilde\partial(\rootv) \to X$ is not necessarily injective.
In particular, it means that $\partial(u)=\bigcup_{v\in\p^{-1}(u)}\partial(v)$, but the union is not necessarily disjoint.
We (partially) remedy this problematic aspect 
by using \emph{partial boundaries}.

\begin{proposition} \label{prop:partial-boundary}
Let $\Ntree$ be  a $\kappa$-net-tree over a metric space $(X,d)$.
There exists a partial boundary $\pb=\pb_\Ntree:\Ntree \to 2^X$ with the following properties:
\begin{enumerate}[label=(\Alph*)]
\item $\pb(u)\subseteq \partial(u)$, for every $u\in \Ntree$.
\item $\pb(\rootv)=\partial(\rootv)=\Image(\Ntree)$.
\item \label{it:disj-union}
The partial boundary of a vertex $u$
 is a \emph{disjoint union} of the partial boundary of its children, i.e.
\begin{equation}\label{eq:pb-dis-union}
\pb(u)=\biguplus_{v\in\p^{-1}(u)}\pb(v).
\end{equation}
\item Every partial boundary $\pb(u)$ is a Borel subset.
\item \label{it:partial-packing}
$B^o_d(\rep{u},\Delta(\p(u))/\kappa)\cap \Image(\Ntree)\subseteq \pb(u)$, for every $u\in\Ntree$.
\end{enumerate}
\end{proposition}
\begin{proof}
The construction of the partial boundary is done inductively
on the net-tree from the root downward.
First we define $\pb(\rootv)=\partial(\rootv)=\Image(\Ntree)$.
Assume inductively  that $\pb(u)\subseteq \partial(u)$ has been defined and is a Borel set. 
Let
$\p^{-1}(u)=\{v_1,\ldots, v_m\}$ be the children of $u$ in some
arbitrary order. 
Define 
\[
\pb(v_i)=\pb(u)\cap \bigl(
\partial(v_i) \setminus
\bigl(\partial(v_1)\cup \ldots \cup \partial(v_{i-1})\bigr)
\bigr).
\]
Obviously $\pb(v_i)\subseteq \partial(v_i)$.
Since $\pb(u)$ is a Borel set, and so are $\partial(v_1),\ldots,\partial(v_i)$ (which are closed according to Proposition~\ref{prop:net-tree-prop}),
we conclude that $\pb(v_i)$ is also a Borel set.
Equality~\eqref{eq:pb-dis-union} follows immediately.

Lastly, we prove Item~\ref{it:partial-packing}.
Fix $u\in\Ntree$, and let 
\[B=\Bigl\{v\in \Ntree:\; u\wedge v\notin\{u,v\}
\text{ and } \Delta(v)<\Delta(u)\le \Delta(\p(v)) 
\Bigr\}.
\]
We claim that
$\tilde\partial(\rootv)=\tilde\partial(u)\cup
\bigcup_{v\in B} \tilde\partial(v)$.
Indeed, fix an arbitrary branch $b\in\tilde\partial(\rootv)$. 
If $u\in b$ we are done.
Otherwise let $v\in b$ be the first vertex on $b$ for which 
$\Delta(v)<\Delta(u)$ (there must be such a vertex because the labels along $b$ vanish). 
By definition, $v$ is not the root and $\Delta(\p(v))\ge \Delta(u)$. 
From the monotonicity of the labels, $v$ is not a strict ancestor of $u$, but it is also not a descendant of $u$ (otherwise, $u\in b$). Hence
$u\wedge v\notin\{u,v\}$ and so $v\in B$.
From that and an inductive application of equality~\eqref{eq:pb-dis-union} we deduce that
$\Image(\Ntree)= \pb(u) \cup \bigcup_{v\in B} \pb(v)$.
By the packing property of $\Ntree$,
\[B^o_d(\rep{u},\Delta(\p(u))/\kappa)\cap \Big(\bigcup_{\substack{v\in B}} \pb(v) \Big) \subseteq 
B^o_d(\rep{u},\Delta(\p(u))/\kappa)\cap \Big(\bigcup_{\substack{v\in B}} \partial(v) \Big)=\emptyset.
\]
Therefore,
 $B^o_d(\rep{u},\Delta(\p(u))/\kappa)\cap \Image(\Ntree) \subseteq \pb(u) $.
\end{proof}

Equality~\eqref{eq:pb-dis-union} applied inductively, implies that the partial boundary of a vertex is a disjoint union of the partial boundary of $\delta$-descendants. 
Formally, for $u\in\Ntree$ with $\Delta(u)>0$,
and $\delta\in(0,\Delta(u)]$,
$\pb(u)=\biguplus_{v\in \tilde\partial(u,\delta)} \pb(v)$.

Proposition~\ref{prop:net-tree-prop} stated that the image of a net-tree is a compact metric space. 
The next proposition states that the reverse is also true.

\begin{proposition} \label{prop:net-tree-exists}
Every compact metric space has a  surjective $20$-net-tree over it.
\end{proposition}
\begin{proof}
Fix a compact metric space $(X,d)$. Since $X$ is bounded, by rescaling we may assume without loss of generality that $\diam(X)=1$. 
Let $\tau = 1/4$. 

Recall that $\delta$-net of a metric space $(X,d)$ is a subset $N\subseteq X$ such that $d(x,y)\ge \delta$  for every pair of distinct points $x,y\in N$, and $d(x,N)<\delta$ for every $x\in X$.
Let  $(N_\ell)_{\ell=0}^{\infty}$ be a sequence of $\delta_\ell$-nets of $X$ where 
$\delta_\ell=(1-\tau)\tau^{\ell}$.
For every $\ell\in\{0,1,2,\ldots\}$, and $a\in N_\ell$, we define
a unique vertex $v_{\ell,a}$. 
Define $\mlevel{v_{\ell,a}}=\ell$, and $\rep{v_{\ell,a}}=a$.
When $\ell>0$ the parent of a vertex $v_{\ell,a}$ is
$v_{\ell-1,c}=\p(v_{\ell,a})$, where $c\in N_{\ell-1}$ is the closest point in $N_{\ell-1}$ to $a$,  breaking ties arbitrarily.

Fix $u,v\in \Ntree$.
If  $\p(v)=u$
then by the construction above,
\[d(\rep{v},\rep{u})=\min_{h\in N_{\mlevel{u}}}d(\rep{v},h)\le (1-\tau) \tau^{\mlevel{u}} = \tau^{\mlevel{u}} - \tau^{\mlevel{v}}.\]
By induction and the triangle inequality,
$d(\rep{u},\rep{v})\le \tau^{\mlevel{u}}-\tau^{\mlevel{v}}$
when $u\in \p^{(*)}(v)$.
Define $\Delta(u)=\tau^{\mlevel{u}}$, and
denote $B_u=B_d(\rep{u}, \Delta(u))$.
Using the triangle ineqaulity and the above bound on the distance between representatives, we conclude that if
$u\in \p^{(*)}(v)$ then $B_v\subseteq B_u$.
This, in particular, implies the covering property of $\Ntree$.

Next, we prove surjectivity --- $\Image(\Ntree)=X$. 
Fix $x\in X$,
and for $m\ge \ell\ge 0$ let 
\begin{align*}
N^m_\ell(x)&=\bigl\{\p^{(m-\ell)}(v_{m,a}) \in\Ntree:\;  v_{m,a}\in \Ntree,\  d(x,a)\le {\tau^m} \bigr\} .
\end{align*}
Observe that for a fixed $\ell\ge 0$, 
$(N_\ell^m(x))_{m=\ell}^{\infty}$ is a sequence of finite, non-increasing and non-empty sets.
Indeed, $N_\ell^m(x)\subseteq N_\ell$ which is finite.
It is also non-empty: By construction $N_m$ is $(1-\tau)\tau^m$-net, and hence
there exists a point $a\in N_m$ for which 
$d(x,a)\le (1-\tau)\tau^m \le\tau^m $, and therefore 
$\psi^{(m-\ell)}(v_{m,a})\in N_\ell^m(x)$.
We are left to prove that $N_\ell^{m+1}(x)\subseteq N_\ell^m(x)$ for every $m\ge \ell$.
Fix $u\in N_\ell^{m+1}(x)$. 
There exist $v_{m+1,a}\in\Ntree$ such that
$d(a,x)\le \tau^{m+1}$, and $u=\p^{(m+1-\ell)}(v_{m+1,a})$, so  $x\in B_{v_{m+1,a}}$ 
Denote $\p(v_{m+1,a})=v_{m,c}$.
As we observed above
$B_{v_{m+1,a}}\subseteq B_{v_{m,c}}$, and hence
$d(x,c)\le \tau^{m}$. This means that
$u\in N_{\ell}^m(x)$. 
As this is true for every $u\in N_\ell^{m+1}(x)$, we conclude that 
$N_\ell^{m+1}(x)\subseteq N_\ell^m(x)$.

We can thus define 
$N_\ell^{\infty}(x)=\bigcap_{m=\ell}^{\infty} N_\ell^m(x)\ne \emptyset$.
$N_\ell^{\infty}(x)$ clearly satisfies the following two properties: 
For every $u\in N_\ell^{\infty}(x)$ there exists
$v\in  N_{\ell+1}^{\infty}(x)$ such that $u=\p(v)$; 
and $d(x,\rep{u})\le \tau^{\ell}$.
Hence we can construct a branch $b=(v_0,v_1,\ldots)$
by choosing $v_0=\rootv \in N^\infty_0(x)=\{\rootv\}$, and $v_{i+1}\in N_{i+1}^{\infty}(x)$ a child of $v_i$. 
From the above,
$\lim_{i\to \infty} d(\rep{v_i},x)=0$, and hence $\rep{b}=x$.

Lastly, we prove the  packing property of $\Ntree$.
Let $u,v\in \Ntree$ such that $u\wedge v\notin\{u,v\}$,
and $\ell(v)\ge \ell(u)$.
Let $w\in\Ntree$ such that $\rep{w}\in B^o_d(\rep{u},\Delta(\p(u))/20)=B_d(\rep{u},\tau^{\mlevel{u}}/5)$.
Assume towards a contradiction that  $w\in\T_v$.
Assume first that $\mlevel{u}=\mlevel{v}$. 
Since $u\wedge v\notin\{u,v\}$, we have $v\ne u$.
 Clearly we can rule out $w=v$, since this  directly contradicts the net-tree construction.
Let $v'\in \p^{(*)}(w)\cap \p^{-1}(v)$, be the child of $v$ which is also an ancestor of $w$. Then, by the assumption and the covering property,
\begin{equation*}
d(\rep{u},\rep{v'})\le d(\rep{u},\rep{w})+d(\rep{v'},\rep{w})
\le\bigl(\tfrac{1}{5}+  \tau \bigr)
\tau^{\mlevel{u}}
<\frac12{\tau^{\mlevel{u}}}.
\end{equation*}
By the net property and the triangle inequality
\[
d(\rep{v},\rep{v'})\ge 
d(\rep{v},\rep{u})- d(\rep{u},\rep{v'})
>\tau^{\mlevel{u}}-\frac12 \tau^{\mlevel{u}} =\frac12 \tau^{\mlevel{u}},
\]
which contradicts the construction of $\Ntree$ in which
$d(\rep{v'},\rep{v})\le d(\rep{v'},\rep{u})$.
This proves $B_d(\rep{u},\Delta(\p(u))/20)\cap \{\rep{w}:w\in \T_v\}=\emptyset$ when $\mlevel{u}=\mlevel{v}$.

If $\mlevel{v}>\mlevel{u}$, then let $\tilde v\in \p^{(*)}(v)$ be an ancestor of $v$ such that $\mlevel{\tilde v}=\mlevel{u}$.
Since $v\wedge u\notin\{u,v\}$, necessarily $u\ne \tilde v$. So from the above, 
$B_d(\rep{u},\Delta(\p(u))/20)\cap \{\rep{w}:w\in\T_{\tilde v}\}=\emptyset$. 
Combining it with the assymption that $v\in\p^{(*)}(\tilde v)$ is an ancestor of $\tilde v$, 
we conclude that in this case too, 
$B_d(\rep{u},\Delta(\p(u))/20)\cap  \{\rep{w}:w\in\T_{ v}\}=\emptyset$.
\end{proof}

For compact ultrametrics some net-trees have more structure:

\begin{lemma} \label{lem:um-rooted-tree}
Fix a net-tree $\T$ with the root $\rootv_\T\in\T$ over a compact metric space $(X,d)$. 
The space $(U,\rho)$,
where $U=\tilde\partial(\rootv)$ and 
$\rho(a,b)=\Delta(a\wedge b)$,  is a compact ultrametric.
Furthermore, $\T$ is also surjective $1$-net-tree over $(U,\rho)$ and the mapping $\rep{\cdot}:(U,\rho)\to (X,d)$ is 2-Lipschitz.

In the reverse direction, 
fix a compact ultrametric $(U,\rho)$. 
There exists a surjective $1$-net-tree $\T$ over $U$ for which
$\rho(\rep{a},\rep{b})=\Delta(a\wedge b)$, for every $a,b\in \tilde\partial(\rootv_\T)$.

Furthemore,
$\mathcal{O}_\T=\{\tilde\partial_\T(u): u=\rootv_\T \text{ or } \Delta(u)<\Delta(\p(u))\}$ is the set of open balls in $(U,\rho)$ and the set of closed balls which are not singleton. 
\end{lemma}
\begin{proof}
We begin with the first statement. 
Observe that in a tree, for any three branches 
$a,b,c\in\tilde\partial(\rootv)$
$a\wedge b$ or $b\wedge c$ (or both) must be a weak ancestor of $a\wedge c$, and from the monotonicity of the labels,
\[
\rho(a,c)= \Delta(a\wedge c)\le 
\max\{\Delta(a\wedge b),\Delta(b\wedge c)\}
=\max\{\rho(a,b),\rho(b,c)\}.
\]
Hence $\rho$ is an ultrametric.
Observe that $\T$ is a surjective net-tree over $(U,\rho)$, so by
Proposition~\ref{prop:net-tree-prop}, $U$ is compact.
The $2$-Lipschitz property of $\rep{\cdot}$ follows from the covering property of $\T$ over $X$.

The second and third statements are proved in~\cite[\S2]{MN-skeleton}.
\end{proof}

The above representation of (compact) ultrametrics as  boundaries of net-trees 
is well known and is discussed in the literature 
using assortment of terminologies. 
Examples are \emph{hierarchical well-separated trees} (in, e.g.,~\cite[\S3.1]{BLMN05}),
\emph{boundary of the visual metric over trees}
(in, e.g.,~\cite[\S1]{arcozzi2019ahlfors}), 
and 
\emph{the end space of trees} (in~\cite[\S1]{Hug04} and references therein).

\section{Proof of the ultrametric skeleton theorem}
\label{sec:um-skeleton-regular}

The result of iterative applications of 
Corollary~\ref{cor:bar-ramsey-decomp-2}
on a net-tree of a given compact space,
is a hierarchy described in the following lemma.
For $A\subseteq X$, denote by $\overline{A}$ the topological closure of $A$.

\begin{lemma}
\label{lem:um-skeleton-regular}
Fix a compact $\lambda$-\emph{doubling} metric space $(X,d)$, 
a finite Borel  measure $\mu$ on $X$,
and $t\in\{2,3,4,\ldots \}$.
Then, there exists a net-tree $\T$ with the root $\rootv\in \T$
and Borel subsets $C_u\subseteq X$ 
associated with every $u\in \T$,  
having the following properties.
\begin{enumerate}[label=(\Alph*)]
\item \label{it:root} $C_\rootv=X$.
\item \label{it:laminar} $C_u\subseteq C_{\p(u)}$ for every $u\in\T\setminus\{\rootv\}$.
\item \label{it:Delta} $\Delta(u)= \diam_d(C_u)$.
\item \label{it:separation} If $u, v\in\T$, $u\wedge v \notin\{u,v\}$,  then $d(C_u,C_v)\ge \Delta({u\wedge v})/(16t)$.
\item \label{it:singleton} For every branch $b\in \tilde\partial(\rootv)$, $\bigcap_{v\in b} \overline{C_v}=\{\rep{b}\}$.  
\item \label{it:partialT=T}  
 $\partial(u)\subseteq \overline{C_u}$, for every $u\in \T$.
\item \label{it:B-subset-C_u} $B_d(\rep{u}, c \Delta(\p(u))/t) \subseteq C_u$, for some universal $c>0$, and for every $u\in \T\setminus\{\rootv\}$.
\setcounter{EnumResume}{\value{enumi}}
\end{enumerate}
In particular, defining $\rho(\rep{a§},\rep{b})=\Delta(a\wedge b)$ for  $a,b\in \tilde\partial(\rootv)$, 
\begin{enumerate}[resume*]
\item \label{it:rho}
$\rho$ is an ultrametric on $\Image(\T)$ satisfying 	$d(x,y) \le \rho(x,y)\le 16t \cdot d(x,y)$, 
for every $x,y\in \Image(\T)$.
\setcounter{EnumResume}{\value{enumi}}
\end{enumerate}
Lastly, there exists a  function $\xi:\T\to[0,\infty)$
satisfying:
\begin{enumerate}[resume*]

\item \label{it:xi subadditive} $\xi$ is sub-additive on $\T$. I.e.,
for every vertex $u\in \T$,
$\xi(u)\le \sum_{v\in\p^{-1}(u)} \xi(v)$.

\item  \label{it:xi-mu} 
For every $u\in \T$, 
\begin{equation}
\label{eq:mu-xi}
\mu^{1-1/t}(C_u)\le \xi(u) \le \lambda^{2/t} \mu^{1-1/t}(C_u).
\end{equation}
\end{enumerate}
\end{lemma}

Lemma~\ref{lem:um-skeleton-regular} is very similar to a corresponding lemma in~\cite{mendel2021simple}, but with the crucial addition of Item~\ref{it:B-subset-C_u}, which adds a significant technical complication to the proof.

\begin{proof}[Proof of Lemma~\ref{lem:um-skeleton-regular}]
Fix a surjective $20$-net-tree $\Ntree$ over $X$ with the root $\rootv_\Ntree$, whose existence was established in 
Proposition~\ref{prop:net-tree-exists}.
The rooted tree $\T$, the clusters associated with it, and their representative points  are defined inductively from top to bottom.
For e vertex $u\in\T$, instead of defining $\rep{u}$ and $C_u$ directly on $X$, we will define ``tilde versions",
 $\widetilde{C}_u\subset \Ntree$ and 
 $\trep{u}\in\widetilde C_u$ using vertices of the net-tree $\Ntree$. We will maintain by induction that $\widetilde C_u$ is finite.
Using the tilde versions, we define $\rep{u}=\repT{\Ntree}{\trep{u}}$,  and $C_u=\bigcup_{x\in \widetilde{C}_u}{\pb_\Ntree(x)}$.
Define $\Delta(u)=\diam_d(C_u)$
(which satisfies Item~\ref{it:Delta}). 

The cluster associated with the root of $\T$,  $\rootv\in\T$, is $\widetilde{C}_{\rootv}=\{\rootv_\Ntree\}$, 
and its representative is $\trep{\rootv}=\rootv_\Ntree$, which satisfies Item~\ref{it:root}.

Assume next that the vertex $u\in\T$, the associated cluster $\widetilde{C}_u$, and its representative $\trep{u}$ were defined.
If $\Delta(u)=0$ (i.e., $C_u$ is a singleton), 
we simply define a new vertex $v\in\T$ whose parent
is $u=\p(v)$,  $\widetilde C_v=\{\tilde v\}$,
$\trep{v}=\tilde v$, and $\Delta(v)=0$.

Next we assume that $\Delta(u)>0$, (i.e., $|C_u|>1$). 
Let $s_u=4^{\lfloor\log_4\Delta(u)\rfloor}/(64t)$ and
\[\widetilde Z_u=\bigcup_{y\in\widetilde C_u}
\tvb_\Ntree(y,s_u).
\]
Since (by the inductive hypothesis) $\widetilde C_u$ is finite and so are $\tvb_\Ntree(y,s_u)$  for $y\in \widetilde C_u$, 
their union, 
$\widetilde Z_u$, is also finite.

Informally speaking, both $\widetilde Z_u$, and $\widetilde C_u$ represent  the cluster $C_u\subset X$. 
They both represents $C_u$ as a disjount union of partial boundaries of $\Ntree$'s vertices, but at different scales.
The scale of the vertices in $\widetilde C_u$ is 
$s_{\p(u)}$, while the scale of the vertices in  
$\widetilde Z_u$ is 
$s_u$.

Observe that necessarily $|\widetilde Z_u|>1$ because, by the covering property, for every vertex $x\in \widetilde Z_u$, 
$\diam_d(\pb_\Ntree(x))\le \Delta(v)/(64t)<\diam_d(C_v)$.
Let $\tilde d_u$ be a metric on $\widetilde Z_u$ defined as 
$\tilde d_u(x,y)=d(\repT{\Ntree}{x},\repT{\Ntree}{y})$.
Since $(\widetilde Z_u,\tilde d)$ is a metric induced by a subset of $X$ it is also $\lambda$-doubling.
We also define a (discrete) measure $\tilde\mu_u$ on $\widetilde Z_u$ as 
$\tilde\mu_u(\{x\})=\mu(\pb_\Ntree(x))$.
Since the partial boundaries of unrelated vertices in $\Ntree$ are disjoint, we have 
for any non-root vertex $u\in\T$, 
$\mu(C_u)=\tilde \mu_{\p(u)}(\widetilde C_u)
=\tilde\mu_u(\widetilde Z_u)$.

Apply Corollary~\ref{cor:bar-ramsey-decomp-2} on $(\widetilde Z_u,\tilde d_u,\tilde \mu_u)$ with $\Delta=\Delta(u)+12 s_u$, and let $\widetilde P,\widetilde Q\subset \widetilde Z_u$ the resulting subsets.
Define new vertices $v$ and $w$ as the children of $u$ in $\T$.
Define $\widetilde C_{v}=\widetilde P$ and $\widetilde C_{w}=\widetilde Q$. 
This, in particular,
satisfies Item~\ref{it:laminar}.
Define $\trep{v}\in\widetilde{C}_v$, $\trep{w}\in \widetilde C_w$  arbitrary net-tree vertices in $\widetilde{C}_v$ and $\widetilde{C}_w$, respectively.
We have thus finished describing the inductive construction of $\T$. 
We are left to check that $\T$ is indeed a net-tree over $X$
and prove the rest of the properties of $\T$.

Item~\ref{it:separation}:
 Fix a pair of vertices $v,w\in \T$ for which
$v\wedge w \notin\{v,w\}$. 
Denote $u=v\wedge w$, and let 
 $v'\in \p^{-1}(u)\cap \p^{(*)}(v)$, $w'\in \p^{-1}(u)\cap \p^{(*)}(w)$ be the children of $u$ which are ancestors of $v$, $w$ (respectively).
 By Corollary~\ref{cor:bar-ramsey-decomp-2},
 \[\tilde d(\widetilde{C}_{v'},\widetilde{C}_{w'})\ge \diam_{\tilde d}(\widetilde{C}_u)/8t.\] 
 By the covering property of the net-tree $\Ntree$,
 \[\Delta(u)=\diam_d(C_u)\le \diam_{\tilde d}(\widetilde{C}_u)+ 2s_u
\le \diam_{\tilde d}(\widetilde{C}_u) +\Delta(u)/(32t).
 \]
 Similarly, by the covering property and the triangle inequality
 \[
 d(C_{v'}, C_{w'}) \ge \tilde d(\widetilde{C}_{v'}, \widetilde{C}_{w'}) 
 - 2s_u
 \ge \tilde d(\widetilde{C}_{v'}, \widetilde{C}_{w'}) -  \Delta(u)/(32t).
 \]
 Concatenating the last three inequalities above,
  \[
 d(C_v,C_w)\ge d(C_{v'}, C_{w'}) 
\ge  \Delta(u)/(16t),\]
 which proves Item~\ref{it:separation}.

Item~\ref{it:singleton}:
Fix a branch $b=(\rootv=v_0,v_1,v_2,\ldots)\in\tilde\partial_\T(\rootv)$.
$\overline{C_{v_i}}$ is compact and 
$C_{v_{i+1}}\subseteq C_{v_i}$.
Obviously $\Delta(v_i)=\diam(C_{v_i}) =\diam(\overline{C_{v_i}})$.
We claim that $\Delta(v_i)\searrow 0$. 
Indeed, $\Delta(v_i)$ is non-increasing. 
Suppose towards a contradiction that there exists $\e>0$ such that $\Delta(v_i)\ge \e$, for every $i\in\mathbb N$. 
If there exists $v_i$ which has only one child then by the construction above, $C_{v_i}$ is a singleton,
all its descendants have only one child, 
and by the definition of $\Delta(\cdot)$, $\Delta(v_j)=0$ for $j\ge i$. 
So assume now that every vertex $v_i$ has two children.
Let $u_i$ be the ``other child" of $v_i$, i.e., the child of $v_i$ such that $u_i\ne v_{i+1}$.
Observe that for $i<j$, $u_i \wedge u_j=v_i$.
Therefore,
for every $i\ne j$,
\[
d(C_{u_i},C_{u_j})\ge \Delta({v_{\min\{i,j\}}})/(16t)\ge \e/(16t),
\]
which contradicts the compactness of $X$.
Hence, $\lim_{i\to\infty} \Delta({v_i})=0$. By Cantor's intersection theorem,
 $\bigcap_{v\in b} \overline{C_v}$ is a singleton.
 Since, by construction $\rep{v_i}\in C_{v_i}$, we have
 $\rep{b}=\lim_{v\to b}\rep{v_i} \in \bigcap_{v\in b} \overline{C_v}$ is the unique element.

Item~\ref{it:partialT=T}:
The assertion $\partial(u)\subseteq \overline{C_u}$, is an immediate corollary of Item~\ref{it:singleton}.

Item~\ref{it:B-subset-C_u}:
Assume $u\ne\rootv_\T$, 
By construction, $C_u\supseteq \pb_\Ntree(\trep{u})$, 
where (again by the construction) $\trep{u}\in\widetilde C_u\subset \widetilde Z_{\p(u)}\subset 
\tvb_\Ntree(\rootv_\Ntree,s_{\p(u)}))$.
This means that 
$\Delta_\Ntree(\p_\Ntree(\trep{u}))\ge \Delta(\p(u))/(256t)$.
Therefore, by the packing property of the net-tree $\Ntree$
(Proposition~\ref{prop:partial-boundary}, Item~\ref{it:partial-packing}), 
we have
\[
C_u\supseteq \pb_\Ntree(\trep{u}) 
\supseteq B^o_d\bigl(\rep{u}, \Delta_\Ntree(\p_\Ntree(\trep{u}))/20\bigr)
\supseteq B_d(\rep{u}, \Delta(\p(u))/(5121t)).
\]

Item~\ref{it:rho}:
Fix $a,b\in \tilde\partial_\T(\rootv)$. 
Since $\rep{a},\rep{b}\in \overline{C_{a\wedge b}}$, we have
$d(\rep{a},\rep{b})\le \diam_d(C_{a\wedge a})=\rho(\rep{a},\rep{b})$.
On the other hand, denote $v\in a\cap\p^{-1}(a\wedge b)$,
 $w\in b\cap\p^{-1}(a\wedge b)$. 
Since $\rep{a}\in \overline{C_{v}}$ while 
$\rep{b}\in  \overline{C_{w}}$
\[ 
d(\rep{a},\rep{b}) \ge d(\overline{C_{v}}, \overline{C_{w}})=  
d({C_{v}}, {C_{w}})
\ge \Delta(a\wedge b)/(16t)=
\rho(\rep{a},\rep{b})/(16t).
\]
This in particular implies that $\T$ is a $(16t)$-net-tree over $X$.

We are left to define 
 $\xi:\T\to[0,\infty)$, and prove that it is sub-additive and approximates $\mu^{1-1/t}(C_u)$.
 Define $\mu^*_u=\tilde \mu_u^{\Delta(u)+12s_u}$, and
\[
\xi(u)=\frac{\mu(C_u)}{\mu_u^*(\widetilde Z_u)^{1/t}}.
\]
If $u$ has only one child then, by construction, $C_u$ is a singleton and $\xi(u)=\mu^{1-1/t}(C_u)$. In this case the sub-additivity is trivial.
So we assume now that $u$ has two children. 
Let $v$ be the  child of $u$ associated with $P$ in Corollary~\ref{cor:bar-ramsey-decomp-2}, and $w$ the child  associated with $Q$.  
By Inequality~\eqref{eq:bar-ramsey-decomp-2} of Corollary~\ref{cor:bar-ramsey-decomp-2}, 
we have
\begin{equation}\label{eq:cor-*}
\xi(u)=\frac{\tilde \mu_u(\widetilde Z_u)}{\tilde\mu_u^{\Delta(u)+12s_u}(\widetilde Z_u)^{1/t}}
\le 
\frac{ \mu ( C_{v})}{\tilde \mu _u^{\Delta(u)/2+6s_u} (\widetilde C_{v})^{1/t}}
+
\frac{\mu(C_{w})}{\mu_u^{\Delta(u)+12s_u}(\widetilde C_w)^{1/t}}
.
\end{equation}

\begin{claim}\label{cl:*-comp}
\begin{align}\label{eq:*-comp}
 \mu_{v}^*(\widetilde Z_{v}) & \le
\tilde \mu _u^{\Delta(u)/2+6s_u} (\widetilde C_{v}),
&
 \mu_{w}^*(\widetilde Z_{w}) & \le
\tilde \mu_u^{\Delta(u)+12s_u}(\widetilde C_w)
\end{align}
\end{claim}
\begin{proof}[Proof of Claim~\ref{cl:*-comp}]
We begin with the inequality on the left.
By Corollary~\ref{cor:bar-ramsey-decomp-2},
\begin{multline}
\label{eq:Delta-v}
\Delta(v)=\diam_d(C_{v})\le \diam_{\tilde d_u} (\widetilde C_{v}) +2 s_u
\\ \le (\tfrac12-\tfrac1{4t})(\Delta(u)+12s_u) +2s_u
\le  \Delta(u)/2 + 6s_u-\Delta(u)/(4t) +2s_u \\
\le \Delta(u)/2  -8 s_u.
\end{multline}

By construction $s_v\le s_u$, and their ratio is a multiple of $4$. 
So either $s_v=s_u$ or $s_v\le s_u/4$.
Assume first that $s_u=s_v$. In this case, by definition,
$\widetilde Z_v=\widetilde C_v$,
and so
\begin{multline*}
\mu_v^*(\widetilde Z_v)
=\tilde \mu_v^{\Delta(v)+12s_v} (\widetilde Z_v)
=\tilde \mu_u^{\Delta(v)+12s_u} (\widetilde C_v)
\\ \stackrel{\eqref{eq:Delta-v}}\le 
\tilde \mu_u^{\Delta(u)/2- 8s_u+12s_u} (\widetilde C_v)
\le \tilde \mu_u^{\Delta(u)/2+6s_u} (\widetilde C_v).
\end{multline*}

Next, assume that $s_v\le s_u/4$.
Fix $x\in \widetilde Z_{v}$.
If  
$y\in B_{\tilde d_{v}}({x}, (\Delta(v)+12s_{v})/4)$, then
\[\tilde d_v(x,y) 
\le \frac{\Delta(v)+12s_v}{4}
\stackrel{\eqref{eq:Delta-v}}\le \frac{\Delta(u)/2-8s_u+12s_v}{4}
\le \frac{\Delta(u)/2-3s_u}{4}
.
\]
By the construction of $\widetilde Z_{v}$, there exists a unique
$x'\in \widetilde C_{v}\cap \p^{(*)}_\Ntree(x)$.
By the covering property of $x'$,
\(
d(\repT{\Ntree}{x},\repT{\Ntree}{x'})\le \Delta_\Ntree(x')\le s_u
\). Similarly, there exists a unique 
$y'\in \widetilde C_{v}\cap \p^{(*)}_\Ntree(y)$ 
and it satisfies 
\(
d(\repT{\Ntree}{y},\repT{\Ntree}{y'})\le s_u
\).
So 
$y'\in B_{\tilde d_{u}}({x'}, (\Delta(u)/2-3s_{u})/4+ 2s_u)$.
This implies 
\[ \tilde \mu_v \bigl(B_{\tilde d_v}(x,(\Delta(v)+12s_v)/4)\cap \widetilde Z_v \bigr) \le
\tilde \mu_u \bigl (B_{\tilde d_u}(x',(\Delta(u)/2+5s_u)/4)\cap \widetilde C_v \bigr),
\]
which implies
\(  \mu_{v}^*(\widetilde Z_{v})  \le
\tilde \mu _u^{\Delta(u)/2+6s_u} (\widetilde C_v)
.\)

\bigskip

We move to prove the right inequality of~\eqref{eq:*-comp}.
Similar to the argument above,
either $s_w=s_u$ or $s_w\le s_u/4$.
Assume first that $s_u=s_w$. In this case, by definition,
$\widetilde Z_w=\widetilde C_w$,
and so
\[
\mu_w^*(\widetilde Z_w)
=\tilde \mu_w^{\Delta(w)+12s_w} (\widetilde Z_w)
=\tilde \mu_u^{\Delta(w)+12s_u} (\widetilde C_w)
\le \tilde \mu_u^{\Delta(u)+12s_u} (\widetilde C_w).
\]
Next, assume that $s_w\le s_u/4$.
Fix $x\in \widetilde Z_{w}$
and 
$y\in B_{\tilde d_{w}}({x}, (\Delta(w)+12s_{w})/4)$. Then
\[\tilde d_w(x,y)
\le \frac{\Delta(w)+12s_w}{4}
\le \frac{\Delta(u)+3s_u}{4}.
\]
By an argument similar to the above,  there exist 
$x'\in \widetilde C_{w}\cap \p^{(*)}_\Ntree(x)$,
$y'\in \widetilde C_{w}\cap \p^{(*)}_\Ntree(y)$
for which
\(
d(\repT{\Ntree}{x},\repT{\Ntree}{x'})\le s_u
\), and 
\(
d(\repT{\Ntree}{y},\repT{\Ntree}{y'})\le s_u
\).
So 
$y'\in B_{\tilde d_{u}}({x}, (\Delta(u)+3s_{u})/4+ 2s_u)$.
This implies 
\[ \tilde \mu_w (B_{\tilde d_w}(x,(\Delta(w)+12s_v)/4)\cap \widetilde Z_w) \le
\tilde \mu_u (B_{\tilde d_u}(x',(\Delta(u)+11s_u)/4)\cap \widetilde C_w),
\]
which implies
\(  \mu_{w}^*(\widetilde Z_{w})  \le
\tilde \mu _u^{\Delta(u)+12s_u} (\widetilde C_w)
.\)
This finishes the proof of Claim~\ref{cl:*-comp}.
\renewcommand{\qed}{}
\end{proof}

Applying~\eqref{eq:*-comp} on~\eqref{eq:cor-*}, we obtain the sub-additivity of $\xi$:
\[
\xi(u) \le 
\frac{ \mu ( C_{v})}{ \mu_{v}^*(\widetilde Z_{v})^{1/t}}
+
\frac{ \mu ( C_{w})}{ \mu_{w}^*(\widetilde Z_{w})^{1/t}}=\xi(v)+\xi(w).
\]
This proves Item~\ref{it:xi subadditive}.

Item~\ref{it:xi-mu}:  
Lastly, we prove Inequality~\eqref{eq:mu-xi}.
Recall that $(\widetilde Z_u,\tilde d_u)$ is also $\lambda$-doubling, so 
using Inequality~\eqref{eq:mu-mu*},
\[
\lambda^{-2}\mu(C_u)=\lambda^{-2}\tilde\mu_u(\widetilde Z_u)
\le \mu^*_u(\widetilde Z_u)
\le \tilde\mu_u(\widetilde Z_u)=\mu(C_u).
\]
Thus, Inequality~\eqref{eq:mu-xi} follows.
\end{proof}

One of the outputs of Lemma~\ref{lem:um-skeleton-regular} is $\xi$, which is essentially a \emph{sub-additive} premeasure on the skeleton $\T$ controlled by $\mu^{1-1/t}$ from above  
{and from below}. 
The next lemma trims $\T$ and ``lowers" $\xi$ to make $\xi$
additive.

\begin{lemma} \label{lem:balanced-subtree}
Fix $\delta\in (0,1]$,
a rooted tree  $\T$ with the root $\rootv\in\T$,
and $\xi:\T\to[0,\infty)$.
Assume that $\xi$ is sub-additive,
i.e., $\xi(u)\le \sum_{\p^{-1}(u)} \xi(v)$,
for every $u\in\T$.
Further, assume that 
for any vertex $v\in\T$ and an ancestor $u\in \p^{(*)}(v)$,
 $\xi(u) \ge \delta \xi(v)$.
Then there exists a subtree $\Stree\subseteq \T$ with the same root $\rootv$, and  
\emph{additive} mapping $\sigma:\Stree\to[0,\infty)$   for which $\sigma(\rootv)=\xi(\rootv)$, 
and $\sigma$ is a $(\delta/2)$-approximation of $\xi$ on $\Stree$. 
I.e., for every $u\in\Stree$,
\begin{equation}\label{eq:sigma-additive}
 \xi(u)\ge \sigma(u) >\frac{\delta}2 \xi(u),\qquad \sigma(u)= \sum_{\p^{-1}(u)\cap \Stree} \sigma(v).
\end{equation}
\end{lemma}
\begin{proof}
We use an inductive argument on $\T$, starting from the root, to decide which vertices to keep in $\Stree$, and how to change $\xi$. 
Formally, the proof is by induction on the depth of the vertices in $\T$. 

During the inductive process we maintain for each vertex of $\T$ a status in $\Stree$ as one of the following: ``pending", ``retained" (in $\Stree$), and ``deleted" (from $\Stree$).
The function $\sigma$ is constructed during the process, and at each point during the process, 
$\sigma$ is defined on the retained vertices
and their non-deleted children.
In the process we maintain the invariants  that the set of vertices which are either ``retained"  or ``pending" form a subtree rooted at $\rootv$, and the retained vertices form 
a (graph theoretic) connected component containing the root $\rootv$.
Furthermore, if a retained vertex $u\in\Stree$ has one or more non-deleted sibling or $u=\rootv$,  then 
$\xi(u)\ge \sigma(u) >\xi(u)/2$. 
Otherwise ($u\ne \rootv $ and $|\p^{-1}(\p(u))\cap \Stree|=1$),
$\xi(u)\ge \sigma(u) \ge \delta \xi(u)/2$.

We begin with all vertices of $\T$ in a ``pending" status.
The mapping $\sigma$ is defined only on the root, $\sigma(\rootv)=\xi(\rootv)$.
In the inductive step, let $u\in \T$ be a vertex in ``pending" status with minimal depth among the pending vertices.
The vertex $u$ is either $\rootv$ (if it is the first step of the inductive process), or its parent is already ``retained".
Either way, by the inductive invariant, $\sigma(u)$ is already set.
The status of $u$ is changed to ``retained".

Let $L\subseteq \p^{-1}(u)$ be a minimal subset (with respect to containment)  of the children of $u$ such that
\(
\sigma(u)\le \sum_L\xi(v)
\) (such a subset exists from the sub-additivity of $\xi$, and the inductive assumption $\sigma(u)\le \xi(u)$).
We label all the vertices in $\p^{-1}(u)\setminus L$ and their descendants as ``deleted" (i.e., they will not be part of $\Stree$).

If $|L|>1$, then from
the minimality condition of $L$, we  have
\[
\frac{\sum_L\xi(v)}2\le \sum_L\xi(v) - \min_L \xi(v)
< \sigma(u). 
\]
Hence $\sum_L\xi(v) < 2\sigma(u)$, and we define
for every $v\in L$,
\[\sigma(v):=\frac{\sigma(u)}{\sum_L\xi(v)}\cdot \xi(v) \in \bigl ( \xi(v)/2,\xi(v)\bigr ].\]
This, in particular, means that, $\sigma(u)=\sum_L\sigma(v)$.

If $L=\{v\}$ (a singleton), we set $\sigma(v)=\sigma(u)$.
The additivity  condition is trivially true. 
Let $w\in\p^{(*)}(u)$ be the lowest ancestor of $u$ which is either $\rootv$ or having at least two non-deleted children.
Let $\{y\}=\p^{(*)}(u)\cap \p^{-1}(w)$ be a (weak) ancestor of $u$ which is a child of $w$.
From the inductive hypothesis about the additivity of $\sigma$, $\sigma(v)=\sigma(u)=\sigma(y)$.
Also from the inductive hypothesis, $\sigma(y)>\xi(y)/2$;
and from the hypotehsis of the lemma, $\xi(y)\ge \delta \xi(v)$.
Hence, $\sigma(v)> \delta \xi(v)/2$.

Since every vertex of $\T$ is of finite depth
and the set of vertices at the same or smaller depth is also finite, 
this
process classifies all the vertices of $\T$ as either ``deleted" or ``retained".
From the inductive invariant, 
the retained vertices forms a subtree $\Stree$, and $\sigma$ is defined on $\Stree$ and satisfies the assertions of the lemma.
\end{proof}

\begin{proof}
[Proof of Theorem~\ref{thm:um-skeleton-regular}]
First observe that when $t=1$ the assertion is trivial: We set $S=\{x\}$ for some $x\in X$, and $\mu=\delta_x$ the dirac delta measure on $x$. 
Henceforth, we therefore assume  $t\in\{2,3,\ldots\}$.

An application of Lemma~\ref{lem:um-skeleton-regular} 
on $X$
results in a net-tree   $\T$,
clusters $\{C_u\}_{u\in\T}$,
and a sub-additive $\xi:\T\to [0,\infty)$
satisfying~\eqref{eq:mu-xi}.
Observe that if $u,v\in\T$, $u\in\p^{(*)}(v)$
then
\[
\xi(u)
\stackrel{\eqref{eq:mu-xi}}\ge 
 \mu^{1-1/t}(C_u)\ge \mu^{1-1/t}(C_v) 
\stackrel{\eqref{eq:mu-xi}}\ge 
\lambda^{-2/t}\xi(v).
\]
Thus, $\T$ and $\xi$ satisfy the conditions of Lemma~\ref{lem:balanced-subtree} with $\delta=\lambda^{-2/t}$.
%
%
Application of Lemma~\ref{lem:balanced-subtree} on $\T$
results in a subtree $\Stree$, and additive $\sigma:\Stree\to [0,\infty)$
that satisfies 
\begin{equation}\label{eq:sigma-mu}
\lambda^{-2/t}\mu^{1-1/t}(C_u)/2 \le 
\lambda^{-2/t}\xi(u)/2    \le    
\sigma(u) \le \xi(u)\le 
\lambda^{2/t} \mu^{1-1/t}(C_u),
\end{equation}
for every $u\in\Stree$.

Define $S=\Image(\Stree)\subseteq \Image(\T)$.
The induced ultrametric on $S$ is the restriction of $\rho$ to $S$.

Recall that a (measure-theoretic) semi-ring in $X$
is a collection of subsets $\mathcal S\in 2^X$ satisfying
\begin{enumerate*}[label=(\roman*)]
\item $\emptyset\in \mathcal S$;
\item if $A,B\in\mathcal S$ then $A\cap B\in\mathcal S$;
\item if $A,B\in\mathcal S$ then there exist $n\ge 0$ and
$A_1,\ldots, A_n\in\mathcal S$ pairwise disjoint such that 
$A\setminus B=\bigcup_{i=1}^n A_i$.
\end{enumerate*}
By Lemma~\ref{lem:um-rooted-tree},
 $\{\emptyset\}\cup \{\partial_\T(v)\}_{v\in\T}$ is 
a semi-ring consisting of all the open balls in $(U,\rho)$. 
Furthermore and
$\nu(\partial_\Stree(v))=\sigma(v)$ is a pre-measure on that semi-ring. 
By Carath\'eodory extension theorem (see~\cite[Theorem~1.53]{Klenke}), $\nu$ can be extended to a measure on the $\sigma$-algebra generated by $\{\partial_\T(v)\}_{u\in T}$, 
which is the $\sigma$-algebra of Borel sets of $(S,\rho)$.
Since the metrics $d|_S$ and $\rho|_S$ are topologically equivalent, and $S$ is a Borel (closed) set of $(X,d)$, $\nu$ can be extended to Borel measure on $X$ by simply define
$\nu(A)=\nu(A\cap S)$ on every Borel $A\subseteq X$.

We next prove~\eqref{eq:measure growth-regular}.
Fix $x\in X$ and $ r\ge 0$.
If $B_d(x,r)\cap S=\emptyset$, then
$\nu(B_d(x,r))=0$ and there is nothing to prove.
Otherwise, let $y\in B_d(x,r)\cap S$, so
\[
B_d(x,r)\cap S \subseteq B_d(y,2r) \cap S \subseteq B_{\rho|_S}(y,32tr).
\]
Since $B_{\rho|_S}(y,32tr)$ is a closed $(S,\rho|_S)$  ball, 
by Lemma~\ref{lem:um-rooted-tree} there exists
some $v\in\Stree$ such that $B_{\rho|_S}(y,32tr)=\partial_\Stree(v)$.
In particular $\diam_d(C_v)=\Delta(v)\le 32tr$.
Observe that 
\[
B_{\rho|_S}(y,32tr)=\partial_\Stree(v)\subseteq \overline{C_v} \subseteq B_d(y,32tr) \subseteq B_d(x,(32t+1)r).
\]
Hence
\begin{multline*}
\nu (B_d(x,r))\le \nu(\partial_\Stree(v))=\sigma(v) \\
\stackrel{\eqref{eq:sigma-mu}}{\le}\lambda^{2/t} \mu(C_v)^{1-1/t}
\le\lambda^{2/t}
  {\mu(B_d(x,(32t+1)r))^{1-1/t}}. 
\end{multline*}

We are left to prove~\eqref{eq:measure-shrink-regular}. 
Let $y\in S$ and $r\ge 0$,
\[
B_d(y,r)\cap S \supseteq B_{\rho|_S}(y,r),
\]
By Lemma~\ref{lem:um-rooted-tree} there exists some $v\in\Stree$
such that $\partial_\Stree(v)=B_{\rho|_S}(y,r)$. 
Let $v\in \Stree$ be such a vertex with the smallest depth in $\Stree$.
In particular $\Delta(v)\le r$ which implies that
$C_v\subset B_d(y,r)$.
Furthermore, either:
\begin{itemize}
\item
 The vertex $v$ is the root. In which case
 $B_d(\rep{v},r/t)\subseteq X=C_v$.
\item 
There exists $z\in \partial_\Stree(\p(v))\setminus B_{\rho|_S}(y,r)$. In this case,
$z,y\in \partial_\Stree(\p(v))$, and $\rho(z,y)>r$, and so 
$\Delta(\p(v))>r$.
 From Item~\ref{it:B-subset-C_u} of Lemma~\ref{lem:um-skeleton-regular}, we also have
\[B_d(\rep{v},cr/t)\subseteq B_d(\rep{v},c\Delta(\p(v))/t)\subseteq C_v \subseteq B_d(y,r),
\]
for some universal constant $c>0$. 
\end{itemize}

We conclude:
\begin{multline*}
\nu(B_d(y,r))\ge \nu(\partial_\Stree(v))=\sigma(v)
\stackrel{\eqref{eq:sigma-mu}}\ge (\lambda^{-2/t}/2) \cdot \mu(C_v)^{1-1/t} \\
\ge
(\lambda^{-2/t}/2) \cdot \mu(B_d(\rep{v},cr/t))^{1-1/t}.
\qedhere
\end{multline*}
\end{proof}

\section{Proof of the Dvoretzky-type theorem}
\label{sec:dvo-proof}

Using Theorem~\ref{thm:um-skeleton-regular}
we prove the following Dvoretzky-type result for Ahlfors regular spaces.
\begin{lemma}
 \label{lem:dvo-Hausdorff-Ahlfors}
 For every $t\in\{2,3,\ldots\}$,  $\alpha>0$, 
 and compact Ahlfors $\alpha$-regular space $(X,d)$, 
 there exists a compact $((1-1/t)\alpha)$-regular subset $Y\subseteq X$ whose ultrametric distortion is at most $O(t)$.
 \end{lemma}
\begin{proof}
Let $\mu$ be the measure satisfying~\eqref{eq:def:ahlfors-regular} for $X$.
By Theorem~\ref{thm:um-skeleton-regular}, there exists
a compact subset $Y\subset X$ which embeds with distortion $O(t)$ in an ultrametric and a measure $\nu$ supported on $Y$ such that for any $x\in X$ and $r\in (0,\diam_d(X))$, 
there exists $x'\in X$ satisfying
\begin{multline*}
   \nu(B_d(x,r))  \le \lambda^{2/t} (\mu(B_d(x,C_1tr)))^{1-1/t}
   \\ \le  \lambda^{2/t} C(C_1tr)^{(1-1/t)\alpha}= A_{\lambda,t}\cdot r^{1-1/t},
 \end{multline*}
 \begin{multline*}
   \nu(B_d(x,r))  \ge 0.5 \lambda^{-2/t} 
    (\mu(B_d(x',c_1r/t)))^{1-1/t} 
    \\ \ge 
    0.5 \lambda^{-2/t} c (c_1r/t)^{1-1/t}= a_{\lambda,t}\cdot r^{1-1/t}.
\end{multline*}
Since $\nu$ is supported on $Y$, it means that $Y$ is $((1-1/t)\alpha)$-regular.
\end{proof}
The lemma above extracts an Ahlfors $\beta$-regular approximate ultrametric subset \emph{for a sequence of}  $\beta$'s as close as we wish to $\alpha$, 
\emph{but not at any}  $\beta<\alpha$. 
To obtain an Ahlfors $\beta$-regular subset for  arbitrary  $\beta<\alpha$, 
we use the following lemma from~\cite{arcozzi2019ahlfors}.
Here we give a quick and self-contained proof based on 
Lemma~\ref{lem:balanced-subtree}.

\begin{lemma}[{\cite[Theorem~1.4]{arcozzi2019ahlfors}}] \label{lem:um-subset-um}
Let $(U,\rho)$ be a compact Ahlfors $\alpha$-regular ultrametric, and  $\beta\in (0,\alpha]$. 
Then there exists a  $\beta$-regular subset $S\subset U$. 
\end{lemma}
\begin{proof}
Assume $\beta<\alpha$ 
(when $\beta=\alpha$ there is nothing to prove).
Let $\T$ be a net-tree representing $U$ according to Lemma~\ref{lem:um-rooted-tree}.
Denote $\xi:\T\to [0,\infty)$, $\xi(v)=\mu(\partial_\T(v))$.
Since $\mu$ is a measure, $\xi$ is additive and monotone.
I.e., $\xi(u)=\sum_{\p^{-1}(u)}\xi(v)$, and $\xi(v)\le \xi(\p(v))$.
Since $\xi$ is additive, and $\beta/\alpha\in(0,1]$,
$\xi^{\beta/\alpha}$ is sub-additive, and since $\xi$ is monotone, so is $\xi^{\beta/\alpha}$.
By applying Lemma~\ref{lem:balanced-subtree} on $\T$ with $\xi^{\beta/\alpha}$ (and $\delta=1$), 
there exists a subtree $\Stree$ of $\T$, and $\sigma:\Stree\to[0,\infty)$ that is additive and 
$\xi(v)^{\beta/\alpha}\ge \sigma(v)\ge \xi(v)^{\beta/\alpha}/2$, for any $v\in\Stree$. 
Define $S=\Image(\Stree)\subseteq U$.

It follows from Lemma~\ref{lem:um-rooted-tree}
that $\{\partial_\Stree(v)\}_{v\in\Stree}$ is 
a (measure-theoretic) semi-ring consisting of all the open balls in $(S,\rho|_S)$. Thus,  $\nu(\partial_\Stree(v))=\sigma(v)$  is a pre-measure on that semi-ring. 
By Carath\'eodory extension theorem, $\nu$ can be extended to a measure on the $\sigma$-algebra generated by $\{\partial_\Stree(u)\}_{u\in \Stree}$, which are the Borel sets of $(S,\rho|_S)$.

Fix a closed ball $B_{\rho|_S}(x,r)\subseteq S$,
where $x\in S$, $r\in(0,\diam_\rho(S))$.
By Lemma~\ref{lem:um-rooted-tree}, there exists $v\in\Stree$ such that
$\partial_\Stree(v)=B_{\rho|_S}(x,r)$, and let $v$ be such vertex of  least depth. I.e., 
Either $v=\rootv$ or $\Delta(v)<\Delta(\p(v))$.
Observe that 
$v\in\T$ and $\partial_T(v)\subseteq B_\rho(x,r)$.
Therefore,
\[
\nu(B_{\rho|_S}(x,r))\le \xi(v)^{\beta/\alpha}=
\mu(\partial_\T(v))^{\beta/\alpha}\le \mu(B_\rho(x,r))^{\beta/\alpha}\le C^{\beta/\alpha} \cdot r^\beta,
\]

In the other direction, assume first that $v=\rootv$. Then
\[
\nu(B_{\rho|_S}(x,r))\ge 0.5\cdot \xi(\rootv)^{\beta/\alpha}=
0.5\cdot \mu(\partial_\T(\rootv))^{\beta/\alpha}=0.5\cdot 
\mu(U)^{\beta/\alpha}\ge 0.5 c^{\beta/\alpha} \cdot r^\beta.
\]
If $v\ne \rootv$, then by the choice of $v$, there exists
$z\in\partial_\Stree(\p(v))$ such that $\rho(z,x)>r$,
and hence 
 \[\partial_\T(\p(v))\supsetneqq B_{\rho}(x,r) \supseteq \partial_\T(v).\] 
 By Lemma~\ref{lem:um-rooted-tree} this means that $\partial_\T(v)= B_{\rho}(x,r)$, and therefore,
\[
\nu(B_{\rho|_S}(x,r))\ge 0.5\cdot \xi(v)^{\beta/\alpha}=
0.5\cdot \mu(\partial_\T(v))^{\beta/\alpha}=0.5\cdot 
\mu(B_\rho(x,r))^{\beta/\alpha}\ge 0.5 c^{\beta/\alpha} \cdot r^\beta.
\]
We conclude that $S$ is $\beta$-regular.
\end{proof}

\begin{proof}[Proof of Theorem~\ref{thm:cantor-ahlfors}]
Fix a bounded Ahlfors $\alpha$-regular space $(X,d)$. 
$X$ must be compact, see~\cite[Corollary~5.2]{DS-fractals}.
Fix $\beta\in(0,\alpha)$. 
Let $t=\lceil \frac{\alpha}{\alpha-\beta}\rceil\in\{2,3,\ldots\}$.
Observe that $(1-1/t)\alpha \ge \beta$ and 
$t\le 2 \frac{\alpha}{\alpha-\beta}$.
By Lemma~\ref{lem:dvo-Hausdorff-Ahlfors} 
there exists an Ahlfors $((1-1/t)\alpha)$-regular compact subset $Z\subset X$ such that $(Z,d)$ embeds in an ultrametric $(Z,\rho)$ with distortion
$O(t)=O(\alpha/(\alpha-\beta))$.
Alhfors regularity is invariant of biLipschitz isomorphisims,
and therefore $(Z,\rho)$ is also $((1-1/t)\alpha)$-regular.
By Lemma~\ref{lem:um-subset-um} there exists a compact subset 
$Y\subset Z$ such that $(Y,\rho|_Y)$ is $\beta$-regular.
Since is $(Y,d|_Y)$ is biLipschitz isomorphic to
$(Y,\rho|_Y)$ (with distortion $O(\alpha/(\alpha-\beta)))$, it is also $\beta$-regular.
\end{proof}

\section{Remarks} 
\label{sec:remarks}

The basic approach used here for proving Theorem~\ref{thm:um-skeleton-regular} is similar to the one employed in~\cite{mendel2021simple}.
In order to obtain the lower bound~\eqref{eq:measure-shrink-regular} on $\nu$,
we amended that approach as follows:
\begin{enumerate}[label=\alph*)]
\item
We made sure that the clusters created in the hierarchical decomposition  
contain balls of the underlying spaces
(Item~\ref{it:B-subset-C_u} of Lemma~\ref{lem:um-skeleton-regular}). 
We achieved it by applying Corollary~\ref{cor:bar-ramsey-decomp-2} not 
directly on the space $X$, but on a hierarchy of the nets (the net-tree) that ensured that  each point of the net carries with it a porpotional ball.

\item Lemma~\ref{lem:um-skeleton-regular}
produced a sub-additive ``premeasure" $\xi$, while we need an \emph{additive} premeasure. 
This issue was addressed in~\cite{MN-skeleton,mendel2021simple}
by artificially lowering $\xi$ down the net-tree.
The resulting measure $\nu$ can not be bounded from below by $\mu^{1-1/t}$.
Here we used a different tactic in Lemma~\ref{lem:balanced-subtree}: 
The bulk of the slack in the sub-additivity is eliminated by trimming the net-tree.
The lowering of $\xi$ is done only sparingly for fine-tuning.
\end{enumerate}

Both steps do not seem to be particularly tied to the construction 
from~\cite{mendel2021simple}, 
and might as well work with the more general construction of ultrametric skeleton from~\cite{MN-hausdorff,MN-skeleton}.
That is, an appropriate version of Theorem~\ref{thm:um-skeleton-regular} may also hold
without the doubling assumption.
We did not pursue this direction further here.

The lower bound~\eqref{eq:measure-shrink-regular} on $\nu$ in Theorem~\ref{thm:um-skeleton-regular} is 
qualitatively weaker than the upper bound~\eqref{eq:measure growth-regular}. 
Some of it is clearly unavoidable: 
the universal quantifier can not be ``for every $y\in X$",
since ultrametric subsets must have ``gaps", like the interval $[0.4,0.6]$ in the Cantor set (as a subset of $X=[0,1]$). 
It is not clear whether the two balls in~\eqref{eq:measure-shrink-regular} can be concentric. 
I.e., whether it is possible to obtain the statement:
``For every $y\in S$ and $r\in [0,\infty)$, 
$
 \nu\left(B_d(y,r)\right) \ge c'_{\lambda,t} \cdot \mu(B_d(y,c_t r))^{1-1/t} 
$."
 
Theorem~\ref{thm:um-skeleton-regular} is phrased for general finite Borel measures, but is  applied in the proof of Theorem~\ref{thm:cantor-ahlfors} only to regular measures, which are, in particular, doubling measures. 
For doubling measures, one could conceivably replace the use of the auxiliary net-tree in the proof of Lemma~\ref{lem:um-skeleton-regular} with Christ's dyadic-decomposition (see~\cite[Theorem~11]{Christ} or~\cite[Theorem~A]{ARSWB}) to simplify the exposition of the proof of Theorem~\ref{thm:cantor-ahlfors}.

\subsection*{Acknowledgments}
The author thanks Assaf Naor for comments on a preliminary version of the paper. 
The author also acknowledges
 support by U.S.-Israel Binational Science Foundation, grant 2018223.

\bibliographystyle{plainurl}
\bibliography{simple-ramsey} 

\end{document}